\documentclass{amsart}
\usepackage{amsmath, amssymb, mathrsfs, verbatim, multirow}
\usepackage[all]{xy}

\newtheorem{Teo}{Theorem}[section]
\newtheorem{Prop}[Teo]{Proposition}
\newtheorem{Lema}[Teo]{Lemma}

\newtheorem{Claim}[Teo]{Claim}

\newtheorem{theorem}{Theorem}[section]
\newtheorem{lemma}[theorem]{Lemma}

\theoremstyle{definition}
\newtheorem{Def}[Teo]{Definition}

\newtheorem{Obs}[Teo]{Remark}

\newtheorem{Pro}[Teo]{Problem}

\newcommand{\N}{\mathbb{N}}

\newcommand{\Llr}{\Longleftrightarrow}
\newcommand{\lra}{\longrightarrow}
\newcommand{\Lra}{\Longrightarrow}
\newcommand{\VR}{\mathcal{O}}
\newcommand{\PI}{\mathfrak{p}}
\newcommand{\MI}{\mathfrak{m}}
\newcommand{\QI}{\mathfrak{q}}
\newcommand{\NI}{\mathfrak{n}}
\newcommand{\hei}{\mbox{\rm ht}}
\newcommand{\QF}{\mbox{\rm Quot}}
\newcommand{\Gl}{\mbox{\rm Gal}}

\newcommand{\Sp}{\mbox{\rm Spec}}

\newcommand{\Spec}{\mathrm{Spec}\,}
\newcommand{\sep}{\;|\;}
\newcommand{\Rad}{\mathrm{Rad}\,}

\begin{document}
\title{Henselian elements}
\author[Kuhlmann]{Franz-Viktor Kuhlmann}
\author[Novacoski]{Josnei Novacoski\\
\mbox{ }\\
\tiny{\bf with an appendix by
Hagen Knaf}}


\address{Department of Mathematics and Statistics,
University of Saskatchewan,
Saskatoon, \newline \indent
SK S7N 5E6, Canada}
\email{fvk@math.usask.ca}

\address{Department of Mathematics and Statistics,
University of Saskatchewan,
Saskatoon, \newline \indent
SK S7N 5E6, Canada}
\email{jan328@mail.usask.ca}

\thanks{During the work on this paper, the first author was partially supported by a Canadian NSERC grant and a sabbatical grant from the University of Saskatchewan. The second author was supported by a graduate research fellowship at the University of Saskatchewan.}

\begin{abstract}
Henselian elements are roots of polynomials  which satisfy the conditions of Hensel's Lemma. In this paper we prove that for a finite field extension $(F|L,v)$, if $F$ is contained in the absolute inertia field of $L$, then the valuation ring $\VR_F$ of $(F,v)$ is generated as an $\VR_L$-algebra by henselian elements. Moreover, we give a list of equivalent conditions under which $\VR_F$ is generated over $\VR_L$ by finitely many henselian elements. We prove that if the chain of prime ideals of $\VR_L$ is well-ordered, then these conditions are satisfied. We give an example of a finite valued inertial extension $(F|L,v)$ for which $\VR_F$ is not a finitely generated $\VR_L$-algebra. We also present a theorem that relates the problem of local uniformization with the theory of henselian elements.

\end{abstract}

\maketitle
\section{Introduction}
For an extension $A\subseteq B$ of rings with unity, an element $b\in B$ is called a \textbf{henselian element} over $A$ if there exists a polynomial $h(X)\in A[X]$ (not necessarily monic) such that $h(b)=0$ and $h'(b)$ is a unit of $B$. In this case, $h$ is called a \textbf{henselian polynomial for $b$}. Henselian elements play an important role, implicitly or explicitly, in many problems of valuation theory. The notion of ``\'etale extension" is closely related to the concept of henselian elements.

The problem of local uniformization for a valued function field $(F|K,v)$ turns out to be close to the problem of elimination of ramification. The valued function field $(F|K,v)$ is said to admit \index{Local uniformization!for a valued function field}\textbf{local uniformization} if for every finite set $Z\subseteq \VR_F$ there exists an affine model $V$ of $(F|K,v)$ such that the center $\PI$ of $v$ on $V$ is a regular point and $Z\subseteq \VR_{V,\PI}$. Elimination of ramification asks whether there exists a transcendence basis $T$ of $F|K$ such that $F$ lies in the ``absolute inertia field" $L^i$ of $L=K(T)$ with respect to some extension of $v$ to the separable-algebraic closure of $F$ (see Definition \ref{Defif}).

A possible approach for local uniformization is to prove that the valued function field $(F|K,v)$ admits elimination of ramification with a transcendence basis $T$ for which the valued rational function field $(L|K,v)$ admits local uniformization. Then in order to find a model $V$ of $(F|K,v)$ we can find a convenient model $V'$ of $(L|K,v)$ and extend it via the inertial extension $F|L$. The set $Z$ appearing in the definition of local uniformization plays an essential role in this task. This is because, when finding the model $V'$, we can require that not only elements obtained from the original set $Z$, but also elements needed to generate the extension $F|L$ belong to $\VR_{V',\PI'}$ ($\PI'$ being the center of $v$ on $V'$). Using this approach, Knaf and the first author proved that every ``Abhyankar valuation" admits local uniformization (see \cite{KK}) and that every valuation admits local uniformization in a finite separable extension of the function field (see \cite{KK_1}). Here, we use this approach to prove our Theorem \ref{Teo_4} (and consequently, also Theorem \ref{Teo_5}) below.

Since an algebra essentially generated by henselian elements over a regular ring is regular (see Proposition \ref{Spiv}), it is important to answer the following:

\begin{Pro}
Take a valued field extension $(F|L,v)$ such that the field $F$ lies in the absolute inertia field $(L^i,v)$ of $L$ and $[F:L]<\infty$ (for short we will call this \textbf{a finite valued inertial extension}). Can we find a generator of $F$ over $L$ which is a henselian element? If that is the case, what can be said about the valuation rings? For instance, is $\VR_F$ generated as an $\VR_L$-algebra by henselian elements?
\end{Pro}

In slightly different terms, these questions were posted by the first author on ``The Valuation Theory Home Page" in form of conjectures (see \cite{FVKC}). On the same web page, Roquette and van den Dries (see \cite{Roqp} and \cite{Vdr}) gave interesting answers to these problems. A part of this paper is dedicated to summarize those answers and extend them to more general settings.

For a field $L$, we will denote by $L^{\mbox{\tiny\rm sep}}$ its separable-algebraic closure, and by $\tilde{L}$ its algebraic closure.
We are interested in working with a finite valued inertial extension $(F|L,v)$, but some of our results work for the following more general case:
\begin{equation}                           \label{sit}
\left\{\begin{array}{ll}
A & \mbox{an integrally closed local domain with quotient field $L$,}\\
\widetilde{\MI} & \mbox{a maximal ideal  of the integral closure $I(A)$ of $A$ in $L^{\mbox{\tiny\rm sep}}$,}\\
F & \mbox{a finite extension of $L$ lying in the inertia field of $L^{\mbox{\tiny\rm sep}}|L$ w.r.t.\ $\widetilde{\MI}$,}\\
A^* & =F\cap I(A) \mbox{ the integral closure of $A$ in $F$,}\\
B & =A^*_{\widetilde{\MI}\cap A^*}\>.
\end{array}\right.
\end{equation}

The next theorem was proved by the first author in \cite{FVKC} for the valuative case, i.e., when $B=\VR_F$ and $A=\VR_L$ for a finite valued inertial extension $(F|L,v)$.

\begin{Teo}\label{main1}
Assume the situation described in (\ref{sit}).
Then there is $\eta\in B$ such that $F=L(\eta)$, the (monic) minimal polynomial $h(X)$ of $\eta$ over $L$ lies in $A[X]$ and $\eta$ and $h'(\eta)$ are units in $B$. In particular, $\eta$ is a henselian element over $A$.
\end{Teo}

For applications, for instance to the local uniformization problem, it is important to know whether for a finitely generated $A$-algebra $R$ lying in $B$ there exists a unit $u$ of $B$ in $A[\eta]$ such that $R\subseteq A[\eta,1/u]$. Our next theorem answers this question to the affirmative. It was proved originally by van den Dries in \cite{Vdr} for the valuative case; our proof for the general case is a straightforward adaptation.

\begin{Teo}\label{main2}
With the situation as in (\ref{sit}) and $\eta$ the element of Theorem \ref{main1}, we have that $B=A[\eta]_\mathfrak n$ where $\mathfrak n=A[\eta]\cap \widetilde{\MI}$. In particular, for every finite set $Z\subseteq B$ there exists a unit $u$ of $B$ in $A[\eta]$ such that $Z\subseteq A[\eta,1/u]$.
\end{Teo}

Another important question is whether we can replace the element $1/u$ obtained in Theorem \ref{main2} by henselian elements. This was answered by Roquette in \cite{Roqp} for the valuative case:

\begin{Teo}\label{main3}
With assumptions and notations as in Theorem \ref{main1} and the extra assumption that $A$ is a valuation ring we have that for every element $b\in B$ there exist henselian elements $r,s\in B$ such that $b\in A[\eta,r,s]$.
\end{Teo}

A natural question is whether the elements $\eta,r,s$ can be chosen independently of the element $b\in B$, in particular, is $B$ a finitely generated $A$-algebra? We show that this is not always true, even for the valuative case. Namely, we prove the following:

\begin{Teo}\label{counterexample}
There exists a finite valued inertial extension $(F|L,v)$ such that $\VR_F$ is not a finitely generated $\VR_L$-algebra.
\end{Teo}

The next theorem gives a list of equivalent conditions under which $\VR_F$ is a finitely generated $\VR_L$-algebra.

\begin{Teo}\label{main4}
Let $(F|L,v)$ be a finite inertial extension. Let $\eta$ be the element obtained in Theorem \ref{main1}. We set $\mathfrak n:=\MI_F\cap \VR_L[\eta]$. Then the following conditions are equivalent:
\begin{description}
\item[(i)] $\VR_F$ is a finitely generated $\VR_L$-algebra;
\item[(ii)] there exists a unit $u$ of $\VR_F$ in $\VR_L[\eta]$ such that $\VR_F=\VR_L[\eta,1/u]$;
\item[(iii)] there are henselian elements $r,s\in\VR_F$ such that $\VR_F=\VR_L[\eta,r,s]$;
\item[(iv)] there exists an element $u\in \VR_L[\eta]\setminus\mathfrak n$ such that $u$ belongs to every prime ideal of $\VR_L[\eta]$ not contained in $\mathfrak n$, i.e.,
\[
\bigcap_{\PI\in \mathcal S}\PI\not\subseteq \mathfrak n,
\]
where $\mathcal S=\{\PI\in \Sp(\VR_L[\eta])\mid\PI\not\subseteq \mathfrak n\}$;
\item[(v)] for every chain of prime ideals $(\PI_\lambda)_{\lambda\in \Lambda}$ of $\VR_L[\eta]$, if $\PI_\lambda\not\subseteq \mathfrak n$ for every $\lambda\in \Lambda$, then
\[
\bigcap_{\lambda\in \Lambda}\PI_\lambda\not\subseteq\mathfrak n.
\]
\end{description}
\end{Teo}
\noindent
For the proof of this theorem we need the fact that $\Sp(\VR_L[\eta])$ consists of finitely many chains of prime ideals, which we prove in Proposition~\ref{finchain}. The assertion of that proposition holds more generally; a criterion for the spectrum of an integral extension $S$ of an integrally closed domain $R$ to consist of finitely many chains of prime ideals is given in the appendix.

\par\medskip
In the next theorem, we give a valuation theoretical condition on the valued field $(L,v)$ which implies the above conditions.

\begin{Teo}\label{main5}
Let $(F|L,v)$ be a finite valued inertial extension. Assume also that the chain of prime ideals of $\VR_L$ is well-ordered by inclusion. Then the equivalent conditions \mbox{\rm \textbf{(i) - (v)}} of Theorem \ref{main4} are satisfied.
\end{Teo}

Theorem \ref{main5} is a generalization of part (3) in the main proposition of \cite{Vdr}. There, van den Dries establishes the result when the chain of prime ideals of $\VR_L$ is finite.

The following theorem relates the local uniformization problem with the theory of henselian elements.

\begin{Teo}\label{Teo_4}
Let $(F|K,v)$ be a valued function field such that $v$ is trivial on $K$. Assume that for every finite set $Z\subseteq \VR_F$ there exists a transcendence basis $T$ of $F|K$ and elements $\eta_1,\ldots,\eta_r\in \VR_F$ which are henselian over $K(T)$ such that $(K(T)|K,v)$ admits local uniformization and $Z\subseteq \VR_{K(T)}[\eta_1,\ldots,\eta_r]$. Then $(F|K,v)$ admits local uniformization.
\end{Teo}

As a consequence of the theorems above we obtain the following:

\begin{Teo}\label{Teo_5}
Let $(F|K,v)$ be a valued function field such that $v$ is trivial on $K$. Assume that there exists a transcendence basis $T$ of $F|K$ such that $(K(T)|K,v)$ admits local uniformization and $F\subseteq K(T)^i$. Then $(F|K,v)$ admits local uniformization.
\end{Teo}

In \cite{KK_1}, Knaf and the first author proved a version of Theorem \ref{Teo_5} without assuming that $v$ is trivial on $K$. They use the theory of local \'etale extensions and classical results from algebraic geometry. The advantage of our proof is that it is simpler and uses only tools from valuation theory. The case where $v$ is not trivial on $K$ will be treated in a subsequent paper.

\par\medskip
\textbf{Acknowledgements.} The authors wish to thank Peter Roquette and Lou van den Dries for the permission to use and publish their results, proved in \cite{Roqp} and \cite{Vdr}, respectively. They also thank Hagen Knaf for very helpful suggestions, and Mark Spivakovsky for interesting discussions and for providing a sketch for the proof of Proposition 4.1.

\section{Preliminaries}

We will assume that the reader is familiar with basic facts about valuations and will use them without further reference. For general valuation theory we recommend \cite{Pre}, \cite{Zar_6} and \cite{KVT}, and for basic commutative algebra we suggest \cite{Ath} and \cite{Mat}.

Let $(F,v)$ be a valued field and consider a prime ideal $\PI$ of $\VR_F$. Then the set
\[
\Delta_\PI:=\{\gamma\in v F\mid -va<\gamma<va \mbox{ for all } a\in \PI\}
\]
is a convex subgroup of $v F$. Define the valuation $v_\PI:F^\times\lra v F/\Delta_\PI$ on $F$ by $v_\PI=\pi_{\Delta_\PI}\circ v$ where $\pi_{\Delta_\PI}$ is the canonical projection of $v F$ onto $v F/\Delta_\PI$. We also define the valuation
\[
\overline v_\PI: Fv_\PI^\times \lra \Delta_\PI
\]
on $Fv_\PI=\VR_{v_\PI}/\MI_{v_\PI}$ by setting $\overline v_\PI(x+\MI_{v_\PI})=v(x)$ for $x\in\VR_{v_\PI}$. Then $Fv\simeq (Fv_\PI)\overline v_\PI$ and the valuations $v$ and $v_\PI\circ\overline v_\PI$ are equivalent, so we identify them.

Take a field $F$, a valuation $v$ on $F$ and a ring $R\subseteq F$ with $F=\QF(R)$. The valuation $v$ is said to \textbf{have a center on $R$} if $R\subseteq \VR_v$, i.e., if $v(\phi)\geq 0$ for every $\phi\in R$. The \textbf{center} $\PI$ of $v$ on $R$ is defined by
\[
\PI:=\{\phi\in R\mid v(\phi)>0\}=R\cap \MI_v.
\]
If $R$ is a local ring with maximal ideal $\MI$, then a valuation on $F$ is said to be \textbf{centered at $R$} if $\MI$ is the center of $v$ on $R$ (i.e., $(R,\MI)$ is dominated by $(\VR_v,\MI_v)$). It is easy to see that if $R$ is any ring with $R\subseteq \VR_v\subseteq F$ and $F=\QF(R)$, then $v$ is centered at $R_\PI$.

For a valued field $(F,v)$ and any subfield $L$ of $F$ we write
\[
\VR_L:=\VR_v\cap L\mbox{ \ and \ }\MI_L:=\MI_v\cap L.
\]
A \textbf{valued field extension} $(F|L,v)$ is a field extension $F|L$ together with a valuation on $F$. If the extension $F|K$ is an algebraic function field, then $(F|K,v)$ is called a \textbf{valued function field}.

We will prove now some basic lemmas which will be used in the sequel.
\begin{Lema}\label{Zorn}
Let $R$ be a commutative ring and take elements $f,g\in R$. Assume that $f^n\notin gR$ for every $n\in \N$. Then there exists a prime ideal $\QI$ of $R$ such that $g\in \QI$ and $f^n\notin \QI$ for every $n\in \N$. In particular, for every non-nilpotent element $f$ of $R$ there exists a prime ideal $\QI$ of $R$ such that $f\notin \QI$.
\end{Lema}

\begin{proof}
Define the set
\[
\mathcal S=\{I\subseteq R\mid I\textnormal{ is an ideal of }R,\ g\in I\textnormal{ and }I\cap \{f^n\mid n\in\N\}=\emptyset\}.
\]
By our assumption on $f$ and $g$, we have that $gR\in \mathcal S$. One can prove that the union over any chain of elements in $\mathcal S$ belongs to $\mathcal S$. Thus, we can apply Zorn's Lemma to obtain a maximal element $\mathfrak q$ for $\mathcal S$. We claim that $\mathfrak q$ is a prime ideal. Otherwise, there would exist elements $\alpha,\beta\in R\setminus \mathfrak q$ such that $\alpha\cdot\beta\in \mathfrak q$. By the maximality of $\mathfrak q$ in $\mathcal S$, we have that
\[
(\mathfrak q+\alpha R)\cap\{f^n\mid n\in\N\}\neq\emptyset\neq(\mathfrak q+\beta R)\cap\{f^n\mid n\in\N\}.
\]
Thus, $f^n=p+\alpha r$ and $f^m=q+\beta s$ for some $p,q\in\mathfrak q$, $r,s\in R$ and $m,n\in\N$. This means that
\[
f^{m+n}=(p+\alpha r)\cdot(q+\beta s)=p q+p\beta s+q\alpha r+\alpha\beta r s\in\mathfrak q,
\]
which is a contradiction to $\QI\in\mathcal S$.

For the last statement, we just apply the previous part for $g=0$.

\end{proof}

\begin{Lema}\label{Loc}
Let $R$ be a domain, $\PI$ a prime ideal and $\phi$ an element of $R$ such that $\phi\notin\PI$. Then $R_\phi=R_{\PI}$ (as subsets of $\QF(R)$) if and only if $\phi\in \mathfrak q$ for every prime ideal $\mathfrak q$ of $R$ such that $\mathfrak q\not\subseteq \PI$.
\end{Lema}

\begin{proof}
First observe that if $\phi$ does not belong to a prime ideal $\mathfrak q$, then $R_\phi\subseteq R_\mathfrak q$. Indeed, take an element $x\in R_\phi$. Then $x=f/\phi^n$ for some $f\in R$ and $n\in\N$. Since $\phi\notin \mathfrak q$ and $\mathfrak q$ is a prime ideal we have that $\phi^n\notin\mathfrak q$. Therefore, $x\in R_{\mathfrak q}$.

Assume that $R_\phi=R_\PI$ and take a prime ideal $\mathfrak q$ such that $\mathfrak q\not\subseteq\PI$. This implies that $R_\phi=R_\PI\not\subseteq R_\mathfrak q$. Therefore, by the previous paragraph, we must have that $\phi\in\mathfrak q$.

For the converse, assume that $\phi$ belongs to every prime ideal of $R$ not contained in $\PI$. By the first paragraph and our assumption that $\phi\notin \PI$ we have that $R_\phi\subseteq R_\PI$. Now take an $x\in R_\PI$. This means that $x=f/g$ for some $f,g\in R$ and $g\notin \PI$. If there exists $\psi\in R$ and $n\in \N$ such that $g\cdot \psi=\phi^n$, then
\[
x=\frac fg=\frac{f\cdot\psi}{\phi^n}\in R_\phi.
\]
Suppose that such $\psi$ and $n$ do not exist. Then $gR\cap \{\phi^n\mid n\in\N\}=\emptyset$. Applying Lemma \ref{Zorn} with $\phi=f$, we obtain a prime ideal $\QI$ such that $g\in \QI$ (and hence $\QI\not\subseteq \PI$) and $\phi\notin \QI$. This is a contradiction to our assumption on $\phi$.
\end{proof}

\begin{Lema}\label{12.5.16}
Let $\MI$ be a maximal ideal of a domain $R$. If $S$ is a local ring such that
\[
R\subseteq S\subseteq R_\MI,
\]
then $S=R_\MI$.
\end{Lema}
\begin{proof}
Take an element $x\in \QF(R)$ such that $x\in R_\MI$. Then $x=f/g$ with $f,g\in R$ and $g\notin \MI$. Since $\MI$ is maximal and $g\notin \MI$ we have that the ideal generated by $g$ and $\MI$ is the whole of $R$. Thus there exist $a\in R$ and $b\in \MI$ such that $1=ag+b$. Let $\MI_S$ be the unique maximal ideal of $S$. Since $\MI R_\MI\cap S$ is an ideal of $S$ we must have that $\MI R_\MI\cap S\subseteq \MI_S$, hence $b\in \MI_S$. This implies that $1-b$ is a unit of $S$ and since $ag=1-b$ we have that $1/g=a/(1-b)\in S$. Therefore, $x\in S$ and $S=R_\MI$.
\end{proof}

The following is Lemma 12.5.17 in \cite{vdD}; for the convenience of the reader we will include its proof here.

\begin{Lema}\label{12.5.17}
Let $R$ be an integrally closed domain with fraction field $L$. Let $F|L$ be a finite separable extension and denote by $R^*$ the integral closure of $R$ in $F$. Take $x\in R^*$ such that $F=L(x)$ (note that every generator of $F$ over $L$ can be multiplied by a suitable element of $R$ to obtain a generator that lies in $R^*$). Then the minimal polynomial $f(X)$ of $x$ over $L$ lies in $R[X]$ and $R^*\subseteq f'(x)^{-1}R[x]$.
\end{Lema}

\begin{proof}
Let $k=[F:L]$ and $\sigma_1,\ldots,\sigma_k$ be all the $L$-embeddings of $F$ in $\widetilde{L}=\widetilde{F}$, where we assume that $\sigma_1=\textrm{id}$. Then we have that
\[
f(X)=\prod_{i=1}^k(X-\sigma_i(x)).
\]
Since $x\in R^*$, we know that $\sigma_i(x)$ belongs to the integral closure of $R$ in $\widetilde L$ for $1\leq i\leq k$. This means that every coefficient of $f$ is integral over $R$ and since $R$ is integrally closed we have that $f(X)\in R[X]$.

Define now
\[
g_i(X)=\frac{f(X)}{X-\sigma_i(x)}=\prod_{j\neq i}(X-\sigma_j(x))\in\widetilde{L}[X].
\]

We shall prove that
\[
g_1(X)\in R^*[X].
\]
Since $f(X)\in R[X]$ and $f(X)=g_1(X)\cdot (X-x)$ (in $F[X]$), using Gauss' Lemma we obtain that all coefficients of $g_1$ must belong to $\VR$. Hence, every coefficient of $g_1$ belongs to $R^*$.

If $i>1$, then $g_i(x)=0$ because $X-x$ divides $g_i$. Also, since $f(X)=g_1(X)\cdot (X-x)$ we have
\[
f'(X)=g_1'(X)\cdot (X-x)+g_1(X)
\]
hence $f'(x)=g_1(x)$. If we write $g_1(X)=c_0+\ldots+c_{k-2}X^{k-2}+X^{k-1}$, then
\[
g_i(X)=\prod_{j\neq i}(X-\sigma_j(x))=\prod_{j\neq 1}(X-\sigma_i\sigma_j(x)) =\sigma_i(c_0)+\ldots+\sigma_i(c_{k-2})X^{k-2}+X^{k-1}.
\]

Take $b\in R^*$. Then
\begin{displaymath}
\begin{array}{rcl}
b\cdot f'(x)&=& b\cdot g_1(x)\\
            &=&\displaystyle\sum_{i=1}^k\sigma_i(b)\cdot g_i(x)\\
            &=&\displaystyle\sum_{i=1}^k\sigma_i(b)\cdot\left(\sigma_i(c_0)+\ldots+\sigma_i(c_{k-2})x^{k-2}+x^{k-1}\right)\\
            &=&\displaystyle\sum_{i=1}^k\sigma_i(bc_0)+\ldots+\left(\sum_{i=1}^k\sigma_i(bc_{k-2})\right)x^{k-2}+
            \left(\sum_{i=1}^k\sigma_i(b)\right)x^{k-1}.
\end{array}
\end{displaymath}
Since each coefficient of $b\cdot f'(x)$ is the trace of some element in $F$ they all belong to $L$. On the other hand, they are integral over $R$, hence they belong to $R$. Thus $b\cdot f'(x)\in R[x]$ and consequently, $b\in f'(x)^{-1}R[x]$. Therefore, $R^*\subseteq f'(x)^{-1}R[x]$.
\end{proof}

We will need the following lemma, which is part (iii) of Theorem 9.3 in \cite{Mat}.

\begin{Lema}\label{trans}
Let $A$ be an integrally closed domain with $L$ as its field of fractions, $F|L$ a normal algebraic extension, and $B$ the integral closure of $A$ in $F$. If $\mathfrak p$ is any prime ideal of $A$, then all ideals of $B$ lying over $\mathfrak p$ are conjugate.
\end{Lema}

\par\smallskip
A domain $R$ is called a \textbf{Pr\"ufer} domain if the localization of $R$ at any prime ideal of $R$ is a valuation ring.

\begin{Prop}\label{finchain}
Let $F|L$ be a finite Galois extension of fields and $R$ a valuation ring of $L$. Let $R^*$ be the integral closure of $R$ in $F$. Then for every ring $S$ such that $R\subseteq S\subseteq R^*$, the set $\Sp(S)$ of prime ideals of $S$ consists of finitely many chains.
\end{Prop}
\begin{proof}
Lemma \ref{trans} shows that $R^*$ is a semi-local ring, i.e., it has only finitely many maximal ideals. On the other hand, since $R$ is a Pr\"ufer domain, also $R^*$ is a Pr\"ufer domain (see Theorem 1.2 of \cite{Fuch}). This means that for each maximal ideal $\MI^*$ of $R^*$ the ring $R^*_{\MI^*}$ is a valuation ring. Hence, $\Sp(R^*_{\MI^*})$ is a chain. Since there is an order preserving bijection between the prime ideals of $R^*_{\MI^*}$ and the prime ideals of $R^*$ contained in $\MI^*$ we conclude that the elements of $\Sp(R^*)$ contained in $\MI^*$ form a chain. Therefore, $\Sp(R^*)$ consists of finitely many chains of prime ideals. By the Lying Over Property for $S\subseteq R^*$ we have that $\Sp(S)=\{\QI\cap S\mid \QI\in\Sp(R^*)\}$. Consequently, $\Sp(S)$ also consists of finitely many chains of prime ideals.
\end{proof}

\begin{Prop}\label{12.5.7}
Take a valuation ring $R$, $f(X)\in R[X]$ a monic polynomial and $\PI$ a prime ideal of $R$. Set
\[
R[x]=R[X]/(f),\textnormal{ for } x=X + (f).
\]
For a polynomial $g\in R[X]$, we denote by $\overline{g}$ the polynomial obtained from $g$ by the reduction modulo $\PI$ of its coefficients. Let
\[
\overline f=\prod_{1\leq i\leq r} \overline{f_i}^{e_i},\textnormal{ with }e_i>0
\]
be the factorization of $\overline f(X)\in R/\PI[X]$ into powers of distinct irreducible polynomials $\overline{f_i}$. Then the prime ideals of $R[x]$ lying over $\PI$ are precisely
\[
\QI_i=(\PI,f_i(x))R[x], \quad 1\leq i\leq r.
\]
Moreover, $\QI_i\neq \QI_j$ if $i\neq j$, and $R[x]/\QI_i\cong \left(R/\PI[X]\right)/(\overline{f_i})$.
\end{Prop}

\begin{proof}
Consider the map
\[
\Phi_i:R[x]=R[X]/(f)\lra \left(R/\PI[X]\right)/(\overline{f}_i),
\]
given by $\Phi_i(g(X)+(f))=\overline g(X)+(\overline f_i)$. We will prove that $\ker(\Phi_i)=\QI_i$. We have that
\[
\QI_i=(\PI,f_i(x))R[x]\subseteq \ker(\Phi_i).
\]
On the other hand, if $\Phi_i(\phi(x))=0$, then $g(X):=\phi(X)-f_i(X)h(X)\in \PI R[X]$ for some $h(X)\in R[X]$. Therefore,
\[
\phi(x)=g(x)+f_i(x)h(x)\in (\PI,f_i(x))R[x].
\]
Since $R/\PI$ is a GCD domain, so is $R/\PI[X]$ (see Theorem 14.5 of \cite{Gil}). Consequently, since $\overline{f}_i$ is irreducible over $R/\PI[X]$, we have that $\overline f_i$ is prime in $R/\PI[X]$. Hence, $(\overline f_i)$ is a prime ideal of $R/\PI[X]$ and thus $R/\PI[X]/(\overline{f}_i)$ is a domain. Also, since $\Phi_i$ is surjective we conclude that $\QI_i$ is a prime ideal.

If $j\neq i$, then the image of $f_j(X)$ is not zero in $R/\PI[X]/(\overline{f}_i)$, which implies that $f_j(x)\notin\QI_i$. Therefore, $\QI_i\neq\QI_j$.

It remains to prove that every prime ideal $\QI'$ of $R[x]$ lying over $\PI$ is of the form $\QI_i$ for some $i$. Since
\[
\prod f_i(X)^{e_i}-f(X)\in \PI R[X],
\]
and $f(x)=0$ we have that $\prod f_i(x)^{e_i}$ belongs to $\QI'$. Consequently, for some $i$ we have that $f_i(x)\in\QI'$ and hence $\QI_i=(\PI,f_i(x))R[x]\subseteq \QI'$. Since both $\QI'$ and $\QI_i$ are prime ideals lying over $\PI$ and $R[x]$ is integral over $R$, from the incomparability property we conclude that $\QI_i=\QI'$.
\end{proof}

\section{Henselian Elements}

Let $(A,\MI_A)$ be an integrally closed local domain with quotient field $L$. Fix a maximal ideal $\widetilde{\MI}$ of the integral closure $I(A)$ of $A$ in $L^{\mbox{\tiny sep}}$.

\begin{Def}\label{Defif}
We define the \textbf{henselization of $L$ with respect to $\widetilde{\MI}$}
as the subfield of $L^{\mbox{\tiny\rm sep}}$ consisting of all elements fixed by the group
\[
\{\sigma\in \Gl(L^{\mbox{\tiny\rm sep}}|L)\mid \sigma(\widetilde{\MI})=\widetilde{\MI}\}.
\]
We also define the \textbf{absolute inertia field of $L$ with respect to $\widetilde{\MI}$}
as the subfield of $L^{\mbox{\tiny\rm sep}}$ consisting of all elements fixed by the group
\[
\{\sigma\in \Gl(L^{\mbox{\tiny\rm sep}}|L)\mid \forall x\in I(A):x-\sigma(x)\in\widetilde{\MI}\}.
\]
As $\widetilde{\MI}$ is fixed and there is no danger of confusion, we will denote the above henselization by $L^h$ and the above absolute inertia field by $L^i$.

Let $(L,v)$ be a valued field and take $A$ to be the valuation ring $\VR_v$ of $v$. Choose an extension of $v$ to $L^{\mbox{\tiny\rm sep}}$ and denote it again by $v$. Then the absolute inertia field $L^i$ of $L$ with respect to $v$ is defined to be the absolute inertia field of $L$ with respect to the maximal ideal $\MI_v\cap I(A)$.
\end{Def}

 In what follows, fix a finite extension $F$ of $L$ lying in the absolute inertia field of $L$ with respect to $\widetilde{\MI}$. Let $F'\subseteq L^{\mbox{\tiny\rm sep}}$ be the normal hull of $F$. We set $A'=F'\cap I(A)$ the integral closure of $A$ in $F'$ and $\MI'=\widetilde{\MI}\cap A'$. In view of Lemma \ref{trans}, the maximal ideals of $A'$ are precisely $\MI'_\sigma:=\sigma^{-1}(\MI')$ where $\sigma\in\Gl(F'|L)$. Since $A'$ is integral over $A^*=I(A)\cap F$, the maximal ideals of $A^*$ are precisely $\MI_\sigma:=A^*\cap \MI'_\sigma$. We write $\MI=\MI_{\mbox{\tiny id}}=A^*\cap \MI'$ and consider the local ring $B=A^*_{\MI}$ with its maximal ideal $\MI_B=\MI B$. We also consider the ring $B'=A'_{\MI'}$ with its maximal ideal $\MI_{B'}=\MI'B'$.

\begin{Lema}\label{Lemmanew}
There are $k:=[B/\MI_B:A/\MI_A]$ many automorphisms $\sigma_1,\ldots,\sigma_k\in$ $\Gl(L^\textrm{\rm sep}|L)$ which induce the distinct $A/\MI_A$-embeddings of $B/\MI_B$ in $(A/\MI_A)^{\mbox{\tiny\rm sep}}$ such that for every $\sigma\in \Gl(L^\textrm{\rm sep}|L)$, if $\sigma\neq\sigma_i$ on $F$ for every $i$, $1\leq i\leq k$, then $\MI_\sigma\neq \MI$.
\end{Lema}
\begin{proof}
From the third assertion in Proposition 1.49 of \cite{Ab} we infer that
$F^h=F. L^h$. This field lies in $L^i$ since $L^h\subset L^i$ and by assumption, $F\subset L^i$. Set $A^h=A'_{\MI'}\cap L^h$ and $B^h=A'_{\MI'}\cap F^h$ (these are local domains and we denote their maximal ideals by $\MI_A^h$ and $\MI_B^h$, respectively). Then by Theorem 1.47 of \cite{Ab} $A^h/\MI_A^h=A/\MI_A$ and $B^h/\MI_B^h=B/\MI_B$. As $F.L^h|L^h$ is a finite subextension of $L^i|L^h$, using Theorem 1.48 of \cite{Ab} we obtain $k:=[B/\MI_B:A/\MI_A]$ many automorphisms $\sigma_1,\ldots,\sigma_k\in \textrm{Aut}(L^{{\mbox{\tiny\rm sep}}}|L^h)$ which induce distinct embeddings $\overline\sigma_1, \ldots,\overline\sigma_k$ of $B/\MI_B$ in $(A/\MI_A)^{{\mbox{\tiny\rm sep}}}$ over $A/\MI_A$ (note that $B/\MI_B|A/\MI_A$ is separable by the same theorem since $F\subset L^i$). Since $F.L^h|L^h$ is an inertial extension of henselian fields, the restrictions of $\sigma_1,\ldots,\sigma_k$ to $F.L^h$ are precisely the $k$ many embeddings of $F.L^h$ over $L^h$ in $L^{\mbox{\tiny\rm sep}}$. Their restrictions to $F$ are $k$ many distinct embeddings of $F$ over $L$ in $L^{\mbox{\tiny\rm sep}}$.

We claim that $\MI_{\sigma}\neq \MI$ for $\sigma\neq\sigma_i$, for every $i$, $1\leq i\leq k$. Assume that $\MI_{\sigma}= \MI$ for some $\sigma\in\textrm{Aut}(L^{\mbox{\tiny\rm sep}}|L)$. Then by Lemma~\ref{trans} there is some $\sigma'\in \textrm{Aut}(L^{\mbox{\tiny\rm sep}}|F)$ such that $\sigma'(\widetilde{\MI})=\sigma^{-1}(\widetilde{\MI})$, that is, $\sigma\sigma'(\widetilde{\MI}) =\widetilde{\MI}$. By definition, $\sigma\sigma'$ is thus an element of the decomposition group of the extension $L^{\mbox{\tiny\rm sep}}|L$ with respect to $\widetilde{\MI}$. On the other hand, this decomposition group is the Galois group of $L^{\mbox{\tiny\rm sep}}|L^h$, where $L^h$ is the henselization of $L$ in $L^{\mbox{\tiny\rm sep}}$ with respect to $\widetilde{\MI}$. Therefore, $\sigma\sigma'$ is trivial on $L^h$ and must consequently coincide on $F.L^h$ with $\sigma_j$ for some $j\in\{1,\ldots,k\}$. It follows that $\sigma$ coincides with $\sigma_j$ on $F$. This proves our claim.
\end{proof}

\begin{Lema}\label{Lroq}
Let $\vartheta \in B/\MI_B$ be a nonzero generator of $B/\MI_B$ over $A/\MI_A$ and take $z\in B$ such that $z+\MI_B=\vartheta$. Then we have:
\begin{description}
\item[(i)] If $\sigma z\in \MI_{B'}$ for every $\sigma\in \Gl(F'|L)$ such that $\MI_\sigma\neq \MI$, then $z$ is a henselian element over $A$.
\item[(ii)] If $A$ is a valuation ring and for every $\sigma\in \Gl(F'|L)$ such that $\MI_\sigma\neq \MI$ either $\sigma z\in \MI_{B'}$ or $(\sigma z)^{-1}\in \MI_{B'}$, then $z$ is a henselian element over $A$.
\end{description}
\end{Lema}

\begin{Obs}
The assumption that $A$ is a valuation ring means that we are restricted to the valuative case, i.e., $A=\VR_L$ and $B=\VR_F$ for a finite inertial extension $(F|L,v)$ (and $F'$ being the normal hull of $F$ in $L^{\mbox{\tiny\rm sep}}$). In that case, the assumption that either $\sigma z\in \MI_{B'}$ or $(\sigma z)^{-1}\in \MI_{B'}$ is equivalent to saying that $\sigma z$ is not a unit of $B'$.
\end{Obs}

\begin{proof}[Proof of Lemma \ref{Lroq}]
Let $\sigma_1,\ldots,\sigma_n$ be all the embeddings of $F$ in $L^{\mbox{\tiny\rm sep}}$ and assume without loss of generality that $\sigma_1,\ldots,\sigma_k$ are the embeddings appearing in Lemma \ref{Lemmanew},
\[
\sigma_j(z)\in \MI_{B'}\quad\mbox{ for }k<j\leq r
\]
and
\[
\sigma_l(z^{-1})\in \MI_{B'}\quad\mbox{ for }r<l\leq n.
\]
Set
\[
h_0(X)=\prod_{i=1}^k(X-\sigma_iz)\prod_{j=k+1}^r(X-\sigma_jz)\prod_{l=r+1}^n(X-\sigma_lz).
\]

In case \textbf{(i)}, we have that $r=n$, and hence all the coefficients of $h_0$ belong to $A=B'\cap L$. Moreover,
\[
h_0(X)+\MI_{B'}[X]=g(X)\cdot X^{n-k}
\]
where $g(X)$ is the minimal polynomial of $\vartheta$ over $A/\MI_A$. Since $\vartheta$ is nonzero, the polynomials $g(X)$ and $X$ are coprime and hence $z+\MI_{B}=\vartheta$ is a simple root of $h_0(X)+\MI_{B'}[X]$. This means that $h_0'(z)$ is a unit of $B'$, and as it belongs to $B$, it is also a unit of $B$.

In case \textbf{(ii)}, we can divide $h_0$ by a suitable coefficient to obtain an \textbf{$A$-primitive polynomial} $h(X)\in L[X]$, i.e., a polynomial having all coefficients in $A$ and at least one of its coefficients being a unit of $A$. Consider
\[
h_1(X)=\prod_{i=1}^k(X-\sigma_iz)\prod_{j=k+1}^r(X-\sigma_jz)\prod_{l=r+1}^n((\sigma_lz)^{-1}X-1),
\]
which is obtained by dividing $h_0(X)$ by the factor $\prod_l \sigma_lz\in \widetilde L$. The polynomials $h$ and $h_1$ differ by a constant factor $d\in \widetilde L$ and since both are $B'$-primitive we have that $d$ is a unit of $B'$. Thus,
\[
h(X)=d\cdot h_1(X)\textnormal{ with }d\in B'^\times
\]
and consequently,
\[
h(X)+\MI_{B'}[X]=(d+\MI_{B'})\cdot g(X)\cdot X^{r-k}\cdot(-1)^{n-r}.
\]
Reasoning as in case \textbf{(i)} we conclude that $z$ is a henselian element over $A$.
\end{proof}

We will prove now our main theorems.

\begin{proof}[Proof of Theorem \ref{main1}]
Let $\vartheta$ be a generator of the separable extension $B/\MI_B|A/\MI_A$. Since $A^*/\MI=B/\MI_B$ we can employ the Chinese Remainder Theorem to find an element $\eta\in A^*$ such that $\eta +\MI_B=\vartheta$ and $\eta\in \MI_\sigma$ for every $\sigma\in \Gl(F'|L)$ such that $\MI_\sigma\neq\MI$ (we can do that because $\MI_\sigma$ is a maximal ideal of $A^*$ for every $\sigma\in \Gl(F'|L)$). Since $\MI_\sigma=A^*\cap\sigma^{-1}(\MI')$ this means that $\sigma(\eta)\in\MI'\subseteq \MI_{B'}$ for every $\sigma\in \Gl(F'|L)$ such that $\MI_\sigma\neq\MI$. Applying part \textbf{(i)} of Lemma \ref{Lroq} we obtain that $\eta$ is a henselian element over $A$ with henselian polynomial
\[
h(X)=\prod_{i=1}^k(X-\sigma_i\eta)\prod_{j=k+1}^n(X-\sigma_j\eta).
\]
It remains to prove that $F=L(\eta)$, i.e., that the polynomial $h(X)$ is irreducible (and is hence the minimal polynomial of $\eta$). It is enough to prove that $\sigma_i\eta\neq \eta$ for $1\leq i\leq n$. From the separability of $B/\MI_B|A/\MI_A$ we have that for $1\leq i\leq k$, $\sigma_i(\vartheta)\neq \vartheta$, hence $\sigma_i(\eta)\neq \eta$. For $i>k$ we have that
\[
\sigma_i \eta+\MI_B=0\neq \vartheta=\eta+\MI_B,
\]
hence $\sigma_i \eta\neq \eta$.
\end{proof}

\begin{proof}[Proof of Theorem \ref{main2}]
Set $\mathfrak n=\MI_B\cap A[\eta]$ and let $h(X)\in A[X]$ be the monic polynomial constructed in the proof of Lemma \ref{Lroq}. Using Lemma \ref{12.5.17} for the first inclusion, we obtain that
\[
A^*\subseteq \frac{1}{h'(\eta)}A[\eta]\subseteq A[\eta]_\mathfrak n\subseteq A^*_{\MI^*}=B.
\]
Hence, by Lemma \ref{12.5.16} we have that
\[
A[\eta]_\mathfrak n = B.
\]

Take any finite set $Z=\{f_1,\ldots,f_r\}\subseteq B$. For every $f_i\in Z\subseteq B=A[\eta]_\mathfrak n$ we have $f_i=\displaystyle\frac{a_i}{b_i}$ for some $a_i,b_i\in A[\eta]$ and $b_i\notin \mathfrak{n}$. Taking
\[
u=\prod_{i=1}^r b_i\in A[\eta]
\]
we obtain that $Z\subseteq A[\eta,1/u]$. As a product of the units $b_i$, also $u$ is a unit of $B$.
\end{proof}

In order to prove Theorem \ref{main3} we will need the following:
\begin{Lema}\label{Roq1}
Let $b\in B$ be a henselian element over $A$ which is a unit. Then $b^{-1}$ is also henselian over $A$.
\end{Lema}

\begin{proof}
Let $h(X)\in A[X]$ be a henselian polynomial for $b$. Define the polynomial
\[
g(X)=X^n\cdot h(X^{-1})\in A[X],
\]
where $n=\deg h$. Then $g\left(b^{-1}\right)=0$. Moreover,
\[
g'(X)=nX^{n-1}h(X^{-1})-X^{n-2} h'(X^{-1})
\]
and hence $g'(b^{-1})=-b^{-(n-2)}h'(b)$. Since both $b$ and $h'(b)$ are units of $B$, so is $g'(b^{-1})$.
\end{proof}

\begin{proof}[Proof of Theorem \ref{main3}]
Since $A$, and consequently also $B$, is a valuation ring by assumption, we are working in the case where $A=\VR_L$ and $B=\VR_F$ for a finite valued inertial extension $(F|L,v)$. Denote by $W$ the set of valuations $w$ on $F$ different from $v$ such that $w|_L=v|_L$.

Take any element $b\in\VR_F$. Take $\vartheta$ and $\eta$ as in the proof of Theorem \ref{main1}.
Then $\vartheta=\eta v$ is a generator of $Fv$ over $Lv$, so $bv$ can be written as a polynomial in $\vartheta$ with coefficients in $Lv$. Let $f(X)\in \VR_L[X]$ be an inverse image of that polynomial under the residue map and define $b'=b-f(\eta)+\eta$. Then $bv=f(\eta)v$ and $b'v=\vartheta$. Also, we have that $b\in \VR_L[\eta,b']$. Therefore, it is enough to prove that $b'\in\VR_L[r,s]$ for some henselian elements $r$ and $s$.

Employing the Chinese Remainder Theorem (see Theorem 6.60 of \cite{KVT}), we find an element $c\in\VR_F$ with the following properties:

$$
\begin{array}{cl}
v(c-1)>0 &\textnormal{ so that }cv=1 \\
 w(c)>0&\textnormal{ if } w(b')\geq 0\\
 w(c)=0&\mbox{ if } w(b')<0
\end{array}
$$
for all $ w\in W$. Then the element $r=b'c$ has the following properties:
$$
\begin{array}{cl}
rv=b'v=\vartheta& \\
 w(r)>0&\mbox{ if } w(b')\geq 0\\
 w(r)<0&\mbox{ if } w(b')<0
\end{array}
$$
and hence $ w(r)\neq 0$ for every $ w\in W$. According to part \textbf{(ii)} of Lemma \ref{Lroq}, $r$ is a henselian element. On the other hand, the element $rc$ has the same above properties, so it is also henselian. Since $v(rc)=0$ we can apply Lemma \ref{Roq1} to obtain that $s:=(rc)^{-1}$ is also a henselian element. Therefore,
\[
b'=r\cdot c^{-1}=r^2\cdot(rc)^{-1}=r^2\cdot s\in\VR_L[r,s]
\]
as required.
\end{proof}

\begin{proof}[Proof of Theorem \ref{main4}]

\textbf{(i)} $\Lra$ \textbf{(ii)}: Since $\VR_F$ is a finitely generated $\VR_L$-algebra, we have that $\VR_F=\VR_L[a_1,\ldots,a_r]$ for some $a_1,\ldots,a_r\in\VR_F$. Applying Theorem \ref{main2} to the set $Z:=\{a_1,\ldots,a_r\}$ we obtain that there exists a unit $u\in \VR_F$ such that $Z\subseteq \VR_L[\eta,1/u]$, with $\eta$ as in Theorem 1.2. Therefore,
\[
\VR_F=\VR_L[a_1,\ldots,a_r]\subseteq \VR_L[\eta,1/u]\subseteq \VR_F
\]
and consequently, equality holds everywhere.

\textbf{(ii)} $\Lra$ \textbf{(iii)}: Just apply Theorem \ref{main3} to the element $a=1/u$.

\textbf{(iii)} $\Lra$ \textbf{(i)}: This is trivial.

\textbf{(ii)} $\Llr$ \textbf{(iv)}: By Theorem \ref{main2}, $\VR_F=\VR_L[\eta]_\mathfrak n$. Further,
$\VR_L[\eta,1/u]=\VR_L[\eta]_{u}$. Applying Lemma \ref{Loc} with $R=\VR_L[\eta]$, $\phi=u$ and $\PI=\mathfrak n$ we obtain that $\VR_L[\eta]_\mathfrak n=\VR_L[\eta]_{u}$ if and only if $u$ belongs to every prime ideal of $\VR_L[\eta]$ which is not contained in $\mathfrak n$. This proves the equivalence.
Note that if $u\in \VR_L[\eta]$ and $\VR_F=\VR_L[\eta,1/u]$, then $u$ is a unit in $\VR_F$.

\textbf{(iv)} $\Lra$ \textbf{(v)}: Suppose that \textbf{(v)} is not true. This means that there exists a chain of prime ideals $(\PI_\lambda)_{\lambda\in \Lambda}$ of $\VR_L[\eta]$ such that $\PI_\lambda\not\subseteq\mathfrak n$ for every $\lambda\in \Lambda$ and
\[
\bigcap_{\lambda\in \Lambda}\PI_\lambda\subseteq\mathfrak n.
\]
Since $(\PI_\lambda)_{\lambda\in \Lambda}\subseteq \mathcal S$,
\[
\bigcap_{\PI\in\mathcal S}\PI\subseteq \bigcap_{\lambda\in \Lambda}\PI_\lambda
\]
and we conclude that \textbf{(iv)} does not hold.

\textbf{(v)} $\Lra$ \textbf{(iv)}: From Proposition \ref{finchain}, $\Sp(\VR_L[\eta])$ consists of finitely many chains of prime ideals, say
\[
\Sp(\VR_L[\eta])=\bigcup_{i=1}^r(\mathfrak q_{i,\lambda_i})_{\lambda_i\in \Lambda_i}.
\]
For each chain $(\mathfrak q_{i,\lambda_i})_{\lambda_i\in \Lambda_i}$, if $\mathfrak q_{i,\lambda_i}\not\subseteq \mathfrak n$ for some $\lambda_i\in \Lambda_i$, then by our assumption
\[
\left(\bigcap_{\mathfrak q_{i,\lambda_i}\not\subseteq\mathfrak n}\mathfrak q_{i,\lambda_i}\right)\not\subseteq \mathfrak n.
\]
Hence, for each $i$ there exits an element $u_i\in \VR_L[\eta]\setminus\mathfrak n$ such that $u_i\in \mathfrak q_{i,\lambda_i}$ for every $\lambda_i\in \Lambda_i$ with $\mathfrak q_{i,\lambda_i}\not\subseteq\mathfrak n$. Take $u$ to be the product of these $u_i$'s. Then $u$ belongs to every prime ideal of $\VR_L[\eta]$ not contained in $\mathfrak n$. Since $\NI$ is the set of elements of $\VR_L[\eta]$ which are not units of $\VR_F$ and $u\notin \NI$, we obtain that $u$ is a unit of $\VR_F$. Therefore, \textbf{(iv)} holds.
\end{proof}

In order to prove Theorem \ref{counterexample} we will need the following result, which is Proposition 4 of \cite{FVKD}.

\begin{Prop}\label{KCE}
There are valued fields $(L,v)$ such that $\VR_L$ contains no minimal nonzero prime ideal and the residue field $(L v_\PI,\overline v_\PI)$ is henselian for every nonzero prime ideal $\PI$, but $(L,v)$ is not itself henselian.
\end{Prop}

\begin{proof}[Proof of Theorem \ref{counterexample}]
Let $(L,v)$ be the valued field given in Proposition \ref{KCE}. Extend $v$ to the algebraic closure $\widetilde L$ of $L$ (denote this extension again by $v$). Since $(L,v)$ is not henselian, there exists a finite extension $F$ of $L$ with $F\subseteq L^h\subseteq (L^i,v)$ and a valuation $ w$ on $F$ with $ w\neq v|_F$ and $v|_L= w|_L$. By Theorem 1.2 there exists $\eta\in F$ such that $F=L(\eta)$, the monic minimal polynomial $h(X)$ of $\eta$ over $L$ lies in $\VR_L[X]$ and $\eta$ and $h'(\eta)$ are units in $\VR_F$. Moreover, from the construction of $h$ in the proof of Theorem 1.2 we have that
\[
hv (X)=X^lg(X)\in Lv[X],\ l\geq 1,
\]
where $g(X)$ is the separable minimal polynomial of a generator $\vartheta$ of $Fv$ over $Lv$.


We note that if $\PI$ is a prime ideal in $\VR_L$ with
$(0)\subsetneq \PI\subsetneq \VR_L$, then $v_\PI$ is nontrivial and $v_\PI\neq v$.

\begin{Claim}\label{CL1}
There exist polynomials $F,G\in \VR_L[X]$ such that $h-FG\in \PI \VR_L[X]$ with $Fv(X)=X^l$, $Gv(X)=g(X)$ and $\deg F=l$.
\end{Claim}

{\it Proof.}
By the assumption on the valued field $(F,v)$ we have that $(Lv_\PI,\overline v_\PI)$ is henselian. Since $X$ and $g(X)$ are coprime, we can infer from part (3) of Theorem 4.1.3 of \cite{Pre} that there exist polynomials $\overline F,\overline G\in \VR_{\overline v_\PI}[X]$ such that
\[
hv_\PI(X)=\overline F(X)\overline G(X),\ \overline F\overline v_\PI=X^l,\ \overline G\overline v_\PI=g(X)\mbox{ and }\deg \overline F=l.
\]
Take any polynomials $F(X),G(X)\in\VR_L[X]$ such that $Fv_\PI=\overline F$, $Gv_\PI=\overline G$ and $\deg F=l$. Then $Fv(X)=X^l$ and $Gv(X)=g(X)$. Since the residue map associated to $v_\PI$ in $\VR_L$ is the reduction modulo $\PI$ and $hv_\PI(X)- Fv_\PI(X) Gv_\PI(X)=0$ we have that every coefficient of $h-FG$ belongs to $\PI$. Therefore, $h-FG\in \PI \VR_L[X]$.
This proves our Claim.

\par\smallskip
Take a prime ideal $\PI$ of $\VR_L$. We claim that there exists a prime ideal $\QI$ of $\VR_L[\eta]$ lying over $\PI$ such that $\QI\not\subseteq \NI=\MI_F\cap \VR_L[\eta]$. By Claim \ref{CL1}, there exist $F,G\in \VR_L[X]$ with $Fv(X)=X^l$ such that $\overline h(X)=\overline F(X)\overline G(X)$ in $\left(\VR_L/\PI\right) [X]$. Take a monic polynomial $f\in \VR_L[X]$ such that its reduction in $\left(\VR_L/\PI\right)[X]$ is irreducible and divides the reduction $\overline F$ of $F$. By Proposition \ref{12.5.7} we get that $\QI=(\PI,f(\eta))\VR[\eta]$ is a prime ideal of $\VR_L[\eta]$ lying over $\PI$. Since $f$ divides $F$ modulo $\PI\VR_L[X]$ it also divides $F$ modulo $\MI_L\VR_L[X]$, so $fv(X)=X^r$ for some $r$, $1\leq r\leq l$. This means that $f(\eta)v=fv(\eta v)=\vartheta^r\neq 0$ and hence $f(\eta)\notin \NI$. Therefore, $\QI\not\subseteq \NI$.

Suppose towards a contradiction that $\VR_F$ is a finitely generated $\VR_L$-algebra. From Proposition \ref{finchain}, $\Sp(\VR_L[\eta])$ consists of finitely many chains of prime ideals, say
\[
\Sp(\VR_L[\eta])=\bigcup_{i=1}^r(\mathfrak q_{i,\lambda_i})_{\lambda_i\in \Lambda_i}.
\]
For each chain $(\mathfrak q_{i,\lambda_i})_{\lambda_i\in \Lambda_i}$, if $\mathfrak q_{i,\lambda_i}\not\subseteq \mathfrak n$ for some $\lambda_i\in \Lambda_i$ we set
\[
\QI_i:=\bigcap_{\mathfrak q_{i,\lambda_i}\not\subseteq\mathfrak n}\mathfrak q_{i,\lambda_i}.
\]
By Theorem \ref{main4}, $\QI_i\not\subseteq \NI$ and in particular, $\QI_i\neq (0)$. Set $\PI_i=\QI_i\cap\VR_L$. Because the prime ideals in $\VR_L$ form a chain, there is $i_0\in \{1,\ldots,r\}$ such that $\PI_{i_0}$ is the intersection of $\PI_i$'s. By the choice of $\PI_{i_0}$, if a prime ideal $\QI$ of $\VR_L[\eta]$ is not contained in $\NI$, then $\PI_{i_0}\subseteq \QI\cap\VR_L$. Also, $\PI_{i_0}$ is a prime ideal of $\VR_L$ lying below the nonzero prime ideal $\QI_{i_0}$ of $\VR_L[\eta]$, so $\PI_{i_0}\neq (0)$. Hence, for every prime ideal $\PI$ of $\VR_L$ such that $(0)\subsetneq \PI\subsetneq\PI_{i_0}$ (which exists because of our assumption on the value group of $(L,v)$) and every prime ideal $\QI$ of $\VR_L[\eta]$ lying over $\PI$ we have that $\QI\subseteq\NI$. This is a contradiction to the conclusion of the previous paragraph.
\end{proof}

\begin{proof}[Proof of Theorem \ref{main5}]
Take any chain $(\mathfrak q_\lambda)_{\lambda\in\Lambda}$ of prime ideals of $\VR_L[\eta]$ such that $\mathfrak q_\lambda\not\subseteq \mathfrak n$ for every $\lambda\in \Lambda$. Observe that $\Lambda$ can be seen as a subset of the indexing set $\Lambda'$ of the prime ideals of $\VR_L$. Indeed, for every $\lambda\in \Lambda$ the ideal $\PI_\lambda\cap\VR_L$ is a prime ideal of $\VR_L$. Also, for any two elements $\lambda_1,\lambda_2\in \Lambda$, if $\PI_{\lambda_1}\cap\VR_L=\PI_{\lambda_2}\cap\VR_L$, then by the incomparability property we have that $\PI_{\lambda_1}=\PI_{\lambda_2}$. Since the prime ideals of $\VR_L$ are well-ordered by inclusion the set $\Lambda$ has a minimum $\lambda_0$. This means that
\[
\bigcap_{\lambda\in \Lambda}\PI_\lambda=\PI_{\lambda_0}\not\subseteq \mathfrak n
\]
which proves that the condition \textbf{(v)} of Theorem \ref{main4} holds.
\end{proof}

\section{Local Uniformization}\label{lcd}
The next proposition is essential for the proof of Theorem \ref{Teo_4}.

\begin{Prop}\label{Spiv}
Let $(L,v)$ be a valued field and fix an extension of $v$ to $\widetilde{L}$ (again denoted by $v$). Let $\eta_1,\ldots,\eta_r\in \widetilde{L}$ be henselian elements over $L$ with henselian polynomials $h_i(X)\in L[X]$. Assume that there exists a (finite dimensional) regular local ring $R\subseteq L$ dominated by $\VR_L$ such that $\QF (R)=L$ and $h_i(X)\in R[X]$ for $1\leq i\leq r$. Then
\[
R[\eta_1,\ldots,\eta_r]_{\MI_v\cap R[\eta_1,\ldots,\eta_r]}
\]
is also regular.
\end{Prop}

To prove Proposition \ref{Spiv} we will need the a few basic results. A ring $R$ is said to be \textbf{reduced} if it has no nonzero nilpotent elements.

\begin{Lema}\label{LZDRSLF}
Every zero-dimensional reduced local ring $(R,\MI)$ is a field.
\end{Lema}
\begin{proof}
It is enough to prove that $\MI=\{0\}$. Take $f \in\MI$ and suppose that $f$ is not nilpotent. By Lemma \ref{Zorn}, there exists a prime ideal $\QI$ of $R$ such that $f\notin \QI$. Since $f\in \MI\setminus \QI$ we obtain that $(0)\subseteq \QI\subsetneq \MI$ which is impossible because $R$ is zero-dimensional. Therefore, every element of $\MI$ is nilpotent and since $R$ is reduced, we obtain that $\MI=\{0\}$.
\end{proof}
\begin{Lema}\label{Locaker}
Let $R$ and $S$ be two rings and $\phi:R\lra S$ be a surjective ring homomorphism. If $\PI$ is a prime ideal of $R$ containing $\ker(\phi)$, then $\phi(\PI)$ is a prime ideal of $S$ and $\phi$ induces a surjective ring homomorphism $\Phi:R_\PI\lra S_{\phi(\PI)}$ by setting $\Phi(a/b)=\phi(a)/\phi(b)$ for $a,b \in R$ with $b \notin \PI$. Moreover, $\ker(\Phi)=\{a/b\in R_\PI\mid a,b \in R, b \notin \PI\mbox{ and }a\in \ker(\phi)\}$.
\end{Lema}
\begin{proof}
The fact that $\ker(\phi)\subseteq \PI$ yields that $\phi(R\setminus\PI)=S\setminus \phi(\PI)$. Indeed, the inclusion $S\setminus \phi(\PI)\subseteq \phi(R\setminus\PI)$ follows from $\phi$ being surjective. To prove the other inclusion, take an element $a\in R\setminus \PI$. For any element $b\in R$ with $\phi(b)=\phi(a)$ we have that $a-b\in \ker(\phi)\subseteq \PI$. Hence, $b\notin \PI$ and consequently, $\phi(a)\notin \phi(\PI)$.

The image of any ideal under a surjective ring homomorphism is an ideal. Take elements $a',b'\in S\setminus \phi(\PI)$ and choose $a,b\in R$ such that $a'=\phi(a)$ and $b'=\phi(b)$. Then $a,b\notin \PI$ and hence $ab\notin\PI$. Consequently, $a'b'=\phi(ab)\notin\phi(\PI)$. Therefore, $\phi(\PI)$ is a prime ideal of $S$.

To prove that $\Phi$ is well-defined, we take two pairs $(a,b),(c,d)\in R\times (R\setminus \PI)$ such that $a/b=c/d$ in $R_\PI$. This means that there exists $l\in R\setminus \PI$ such that $l(ad-bc)=0$. Thus $\phi(l)(\phi(a)\phi(d)-\phi(b)\phi(c))=0$ and since $\phi(l)\notin \phi(\PI)$ we conclude that $\phi(a)/\phi(b)=\phi(c)/\phi(d)$ in $S_{\phi(\PI)}$. Also, since $\phi$ is surjective, so is $\Phi$.

If $a\in \ker(\phi)$, then $\Phi(a/b)=\phi(a)/\phi(b)=0$ for every $b\notin \PI$. On the other hand, take $\alpha=c/d\in R_\PI, d\notin \PI$ such that $\Phi(\alpha)=0$. Then there exists $l'\in S\setminus\phi(\PI)$ such that $l'\phi(c)=0$. For each $l\in R$ such that $l'=\phi(l)$ we have that $l\notin\PI$, hence $b:=ld$ belongs to $R\setminus \PI$. Set $a=lc$. Then $c/d=a/b$ in $R_\PI$ and $a\in \ker(\phi)$. Consequently, $\ker(\Phi)\subseteq \{a/b\in R_\PI\mid  a,b \in R, b \notin \PI\mbox{ and }a\in \ker(\phi)\}$. Therefore, $\ker(\Phi)= \{a/b\in R_\PI\mid  a,b \in R, b \notin \PI\mbox{ and }a\in \ker(\phi)\}$
\end{proof}
\begin{proof}[Proof of Proposition \ref{Spiv}]
Let $R^{(s)}=R[\eta_1,\ldots,\eta_s]_{\MI_v\cap R[\eta_1,\ldots,\eta_s]}$ for $1\leq s\leq r$ and $R^{(0)}=R$. Since
\[
R^{(s)}=R^{(s-1)}[\eta_s]_{\MI_v\cap R^{(s-1)}[\eta_s]}\mbox{ for } 1\leq s\leq r,
\]
it is enough to prove our proposition for $r=1$. For this case we denote $\eta_1=\eta$, $h_1=h$ and $R'=R^{(1)}=R[\eta]_{\MI_v\cap R[\eta]}$.

We claim that $\dim (R)=\dim (R')$. Indeed, since $R[\eta]$ is integral over $R$, for every proper chain of prime ideals $C$ of $R$ there exists a proper chain of prime ideals $C_\eta$ of $R[\eta]$ lying over $C$ which is contained in $\MI_v\cap R[\eta]$. On the other hand, by the incomparability property, for each proper chain of prime ideals of $R[\eta]$ their intersections with $R$ form a chain of prime ideals in $R$ of the same length. Hence, $\hei(\MI_v\cap R[\eta])=\dim (R)$. Also, since there exists a bijection between the chains of prime ideals of $R'=R[\eta]_{\MI_v\cap R[\eta]}$ and the chains of prime ideals of $R[\eta]$ contained in $\MI_v\cap R[\eta]$ we conclude that $\hei(\MI_v\cap R[\eta])=\dim(R')$. Therefore, $\dim (R)=\dim (R')$.

Since $R$ is regular, its maximal ideal $\MI$ is generated by $d=\dim (R)$ many elements. We want to prove that also the maximal ideal $\MI'$ of $R'$ is generated by $d$ many elements. To do this, it is sufficient to show that $\MI'$ is generated by the same generators as $\MI$, i.e., that $\MI' = \MI R'$. We have to prove that $\MI R'$ is a maximal ideal, or, equivalently, that $R'/\MI R'$ is a field.

Consider the canonical surjective homomorphism
\[
\psi :R[\eta]= R[X]/(h(X))\lra R/\MI[X]/(\overline h(X))
\]
where $\overline g(X)$ denotes the polynomial in $R/\MI[X]$ obtained from $g(X)\in R[X]$ by reduction of its coefficients modulo $\MI$. Consider the ideal $\mathfrak q=\{\overline g(X)+(\overline h(X))\mid g(\eta)\in\MI_v\}=\psi(\MI_v\cap R[\eta])$ of $\left(R/\MI[X]\right)/(\overline h(X))$. By Lemma \ref{Locaker}, $\QI$ is a prime ideal and $\psi$ induces a surjective homomorphism
\[
\Psi:R'\lra \left(R/\MI[X]/(\overline h(X))\right)_\mathfrak q
\]
whose kernel is $\{f(\eta)/g(\eta)\in R'\mid f,g \in R[X], f(\eta)\in \ker(\psi)\mbox{ and }
g(\eta)\notin \MI_v\cap R[\eta]\}$. Take $\alpha\in \ker(\Psi)$ and write $\alpha=f(\eta)/g(\eta)$ as in the description of this kernel. Then $\psi(f(\eta))=0$ means that $\overline f(X)\in(\overline h(X))$. Hence, $f(X)-h(X)g(X)\in \MI R[X]$ for some $g\in R[X]$. Then $f(\eta)\in\MI R[\eta]$. Write $f(X)=a_0+\ldots+a_rX^r$. Then $a_i\in\MI$ for every $i$, $0\leq i\leq r$. Also, $\eta^i/g(\eta)\in R'$ for every $i$, $0\leq i\leq r$. Hence,
\[
\alpha=\frac{f(\eta)}{g(\eta)}=a_0\cdot\frac{1}{g(\eta)}+\ldots+a_r\cdot\frac{\eta^r}{g(\eta)}\in \MI R',
\]
which proves that $\ker(\Psi)\subseteq \MI R'$. The inclusion $\MI R'\subseteq \ker(\Psi)$ is trivial. Therefore, $R'/\MI R'\cong \left(R/\MI[X]/(\overline h(X))\right)_\mathfrak q$.

We have to prove that $K:=\left(R/\MI[X]/(\overline h(X))\right)_\mathfrak q$ is a field. Since $\left(R/\MI[X]\right)/(\overline h(X))$ is an integral extension of a field, it is zero-dimensional. Hence, $K$ is also zero-dimensional. On the other hand, $K$ is a local ring and thus, by use of Lemma \ref{LZDRSLF}, it remains to show that $K$ is reduced.

Since $h(\eta)=0$ we can write $\overline h(X)=(X-\overline \eta) \overline t(X)$ for some $t(X)\in R[X]$. This means that $h(X)=(X-\eta)t(X)+l(X)$ for some $l(X)\in \MI R[X]$. Then $h'(\eta)=t(\eta)+l'(\eta)$ and since $v(h'(\eta))=0$ we obtain that $v(t(\eta))=0$. Hence, $\overline t(X)+(\overline h(X))\notin\QI$.

Take an element $f(X)\in R[X]$ and assume that $\overline f^n(X)+(\overline h(X))=0$ for some $n\in \N$. This implies that $\overline h(X)$ divides $\overline f^n(X)$ and since $X-\overline \eta$ divides $\overline h(X)$ we obtain that $\overline f(X)=(X-\overline \eta)\overline q(X)$ for some $q(X)\in R[X]$. Thus,
\[
\overline t(X) \overline f(X)=(X-\overline\eta) \overline t(X) \overline q(X)=\overline h(X) \overline q(X),
\]
and hence $\overline t(X)\overline f(X)+(\overline h(X))=0$.

Take now $\alpha=\left(\overline f(X)+(\overline h(X))\right)/\left(\overline g(X)+(\overline h(X))\right)\in K$ such that $\alpha^n=0$ for some $n\in\N$. Let us prove that $\alpha=0$. By definition of localization, there exists $\overline {t_1}(X)+(\overline h(X))\notin \QI$ such that $\overline {t_1}(X)\overline f^n(X)+(\overline h(X))=0$. Then $\left(\overline {t_1}(X)\overline{f}(X)\right)^n+(\overline h(X))=0$ and by the last paragraph, we obtain that $\overline t(X)\overline{t_1}(X)\overline f(X)+(\overline h(X))=0$. Since $\overline t(X)\overline{t_1}(X)+(\overline h(X))\notin \QI$ we conclude, by definition of localization, that $\alpha=0$ in $K$. Therefore, $K$ is reduced, which concludes our proof.
\end{proof}

\begin{proof}[Proof of Theorem \ref{Teo_4}]
Take any finite set $Z\subseteq \VR_F$. We can assume, adding a set of generators of $F$ over $K$ to $Z$ if necessary, that $K(Z)=F$. By assumption, there exist a transcendence basis $T$ of $F|K$ and henselian elements $\eta_1,\ldots,\eta_r$ over $L=K(T)$ such that $Z\subseteq\VR_L[\eta_1,\ldots,\eta_r]$ and $(L|K,v)$ admits local uniformization. Let $Z'\subseteq \VR_L$ be the finite set containing the coefficients of all $h_i$'s and the coefficients of all $\zeta\in Z$ as polynomials in $\eta_1,\ldots,\eta_r$. Since $(L|K,v)$ admits local uniformization, there exists a model $V=\Sp(R)$ of $L|K$ such that $\VR_{V,\PI}=R_{\MI_L\cap R}$ is regular (where $\PI=\MI_L\cap R$ is the center on $V$ of the restriction of $v$ to $L$) and such that $Z'\subseteq \VR_{V,\PI}$. Let $R'=R[\eta_1,\ldots,\eta_r]$ and $V'=\Sp(R')$ (observe that $V'$ is a model of $(F|K,v)$ because $K(Z)=F$). Then by Proposition \ref{Spiv},
\[
\VR_{V',\PI'}=\VR_{V,\PI}[\eta_1,\ldots,\eta_r]_{\MI_F\cap \VR_{V,\PI}[\eta_1,\ldots,\eta_r]}
\]
is regular, where $\PI'=\MI_F\cap R'$ is the center of $v$ on $V'$. Also, since each element $\zeta\in Z$ is a polynomial in $\VR_L[\eta_1,\ldots, \eta_r]$ with coefficients in $\VR_{V,\PI}$ we have that $\zeta\in\VR_{V',\PI'}$. Therefore, $(F|K,v)$ admits local uniformization.
\end{proof}

\begin{proof}[Proof of Theorem \ref{Teo_5}]
Take a finite subset $Z$ of $\VR_L$. By assumption, there exists a transcendence basis $T$ of $F|K$ such that $(K(T)|K,v)$ admits local uniformization and $F\subseteq K(T)^i$. By Theorem \ref{main2} and \ref{main3} there exist henselian elements $\eta,r,s\in\VR_F$ such that $Z\subseteq \VR_{K(T)}[\eta,r,s]$. Hence, we can apply Theorem \ref{Teo_4} to conclude that $(F|K,v)$ admits local uniformization.
\end{proof}

\setcounter{section}{5}
\par\bigskip\bigskip\bigskip\bigskip
\begin{center}
{\large\bf Appendix: Spectra consisting of a finite number of chains}\\
{\it Hagen Knaf}
\end{center}
\par\bigskip
Let $R$ be a commutative ring with $1$. A subset $C\subseteq\Spec (R)$ totally ordered with respect to inclusion is called a \textbf{chain}. For a ring extension $R\subseteq S$ and a chain $C\subseteq\Spec(S)$ the chain $\{P\cap R\sep P\in C\}\subseteq\Spec(R)$ is denoted by $C\cap R$.

The purpose of this note is to provide a proof of the following result:
\begin{theorem}
\label{T:FCP}
Let $R\subseteq S$ be an integral extension of domains such that the corresponding extension $K\subseteq L$ of fraction fields is finite, and assume that $R$ is integrally closed. If $\Spec (R)$ is the union of finitely many chains, then the same is true for $\Spec(S)$.
\end{theorem}
It can be derived from the following result of Kang and Oh, \cite{K-O}:
\begin{theorem}
\label{T:Lift}
Let $R$ be a domain with field of fractions $K$. For every chain $C\subseteq\Spec(R)$ there exists a valuation ring $O\supseteq R$ of $K$ and a chain $C_O\subseteq\Spec(O)$ such that $C_O\cap R=C$.
\end{theorem}
The following two observations are simple but useful:
\begin{lemma}
\label{L:integral_descend}
Let $R\subseteq S$ be an integral ring extension. If $\Spec (S)$ is the union of finitely many chains, then the same is true for $\Spec(R)$.
\end{lemma}
\proof
The restriction map $\Spec(S)\rightarrow\Spec(R)$ preserves inclusions and, due to the Lying-over-Theorem, is surjective.
\qed
\begin{lemma}
\label{L:purely_insep_ascend}
Let $R\subseteq S$ be an integral extension of domains such that the corresponding extension $K\subseteq L$ of fraction fields is purely inseparable, and $R=K\cap S$. Then the restriction map $\Spec(S)\rightarrow\Spec(R)$ is an isomorphism of partially ordered sets.
\end{lemma}
\proof
It suffices to prove that over every prime $P\in\Spec(R)$ lies exactly one prime of $S$. To this end one shows that the radical $Q:=\Rad(PS)$ is prime: assume $st\in Q$ for $s,t\in S$. By assumption there exists $m\in\mathbb{N}$ such that $s^{p^m},t^{p^m}\in K$. Hence $s^{p^m}t^{p^m}\in S\cap K=R$ and thus $s^{p^m}t^{p^m}=(st)^{p^m}\in Q\cap R=P$. Consequently $s\in Q$ or $t\in Q$ as desired.
\qed
\begin{proof}[Proof of Theorem \ref{T:FCP}] It suffices to consider the case of an integrally closed domain $S$: if the theorem is proved in this case one can conclude that the spectrum of the integral closure of $S$ in $L$ is a union of finitely many chains. By Lemma \ref{L:integral_descend} this property then descends to $S$.

Next let $N$ be the normal hull of $L$ over $K$. If the spectrum of the integral closure of $S$ in $N$ is a union of finitely many chains, again by Lemma \ref{L:integral_descend} this holds for $\Spec (S)$ too. Consequently one can assume $K\subseteq L$ to be normal.

Now let $M$ be the fixed field of the automorphism group of the extension $K\subseteq L$, and let $R^\prime$ be the integral closure of $R$ in $M$. By assumption about $R$ and applying Lemma \ref{L:purely_insep_ascend} one gets that $\Spec (R^\prime)$ is a union of finitely many chains. Consequently one can assume $K\subseteq L$ to be Galois.

Let $\Spec(R)=C_1\cup\ldots\cup C_r$, where $C_i$ are chains. By Theorem \ref{T:Lift} for every $i$ there exists a valuation ring $O_i\supseteq R$ of $K$ and a chain $C_{O_i}\subseteq\Spec(O_i)$ such that $C_i=C_{O_i}\cap R$. Let $T_i$ be a valuation domain of $L$ lying over $O_i$. Since the spectra of $O_i$ and $T_i$ are order-isomorphic via the restriction map, there exists a chain $C_{T_i}\subseteq\Spec(T_i)$ with $C_{T_i}\cap O_i=C_{O_i}$. Let $D_i:=C_{T_i}\cap S$ -- note that $S\subseteq T_i$ since $R\subseteq S$ is integral. Then $D_i\cap R=C_i$. For every automorphism $\sigma$ of the Galois group $G$ of $K\subseteq L$ define the chains
\[
D_i^\sigma :=\{\sigma(Q)\sep Q\in D_i\}.
\]
Since the primes lying over a given prime $P\in\Spec(R)$ are $G$-conjugates one gets
\[
\Spec(S)=\bigcup\limits_{\sigma\in G}\bigcup\limits_{i=1}^r D_i^\sigma .
\]
\end{proof}

\end{document}